\documentclass{amsart}
\usepackage[utf8]{inputenc}
\usepackage{amsmath}
\usepackage{amsfonts}
\usepackage{amssymb, amsthm}
\usepackage{stmaryrd}
\usepackage[english]{babel}
\usepackage{graphicx} 
\usepackage[T1]{fontenc}
\usepackage{xcolor}
\usepackage{fancyhdr}
\usepackage{eucal}
\usepackage[all]{xy}
\usepackage{longtable}

\newtheorem{thm}{Theorem}[section]

\newtheorem{lm}[thm]{Lemma}
\newtheorem{cor}[thm]{Corollary}
\newtheorem{prop}[thm]{Proposition}

\newtheorem{defn}[thm]{Definition}

\newtheorem*{notations}{Notations}
\newtheorem*{remarque}{Remark}

\def\R{\mathbb{R}}
\def\Q{\mathbb{Q}}

\def\K{\mathbb{K}}
\def\C{\mathbb{C}}
\def\F{\mathbb{F}}
\def\N{\mathbb{N}}
\def\L{\mathbb{L}}

\def\O{\mathfrak{O}}

\def\house#1{{%
    \setbox0=\hbox{$#1$}
    \vrule height \dimexpr\ht0+1.4pt width .4pt depth \dp0\relax
    \vrule height \dimexpr\ht0+1.4pt width \dimexpr\wd0+2pt depth \dimexpr-\ht0-1pt\relax
    \llap{$#1$\kern1pt}
    \vrule height \dimexpr\ht0+1.4pt width .4pt depth \dp0\relax
}}

\def\degtr{\mathrm{tr.deg}}
\def\Log{\mathrm{Log}}
\def\Hom{\mathrm{Hom}}
\def\GL{\mathrm{GL}}
\def\Norm{\mathrm{Norm}}

\begin{document}

\title{Mahler's method and Carlitz logarithms}
\author{Guillaume Estienne}
\address{Universite Caen Normandie, CNRS, Normandie Univ, LMNO UMR6139, F-14000 CAEN, FRANCE }
\email{guillaume.estienne@unicaen.fr}

\begin{abstract}
In 2007, Papanikolas \cite{Papanikolas_2007} established that if Carlitz logarithms of algebraic functions are linearly independent over the rational function field, then they are algebraically independent.
The purpose of the present paper is to provide a new proof of this theorem using Mahler’s method instead of the theory of $t$-motives.
We revisit and extend the approach developed by Denis in \cite{Denis_2006}, which enabled him in 2006 to prove this result in the particular case of the logarithm of elements in $F_q(\theta)$ via a Mahler system. 
\end{abstract}

\maketitle

\tableofcontents

\section{Introduction}
Throughout this article, we let $k = \F_q(\theta)$ denote the function field in the indeterminate $\theta$ over the finite field $\F_q$ with $q$ elements, where $q$ is a power of a prime number $p$. 
We denote by $v_\infty$ the $\frac{1}{\theta}$-adic valuation on $k$. We also let $\C_\infty$ be a completion of an algebraic closure of $k_\infty := \F_q(\!(\frac{1}{\theta})\!)$, and $\overline{k}$ be the algebraic closure of $k$ contained in $\C_\infty$.

For any element $z\in \C_\infty$, the Carlitz exponential is defined by:
$$\exp_C(z) = \sum_{k \geq 0} \frac{z^{q^k}}{D_k},$$
where $D_k = \prod_{i=1}^k (\theta^{q^i} - \theta)^{q^{k-i}}$. 
This function satisfies the functional equation $\exp_C(\theta z) = \theta \exp_C(z) + \exp_C(z)^q$.
Its kernel is a free $\F_q[\theta]$-module of rank 1, and we denote a generator by $\Tilde{\pi}$. Explicitly, we have:
$$\Tilde{\pi} = \theta\Tilde{\theta} \prod_{k=1}^\infty \left(1- \theta^{1-q^k}\right)^{-1}$$ 
where $\Tilde{\theta}$ denotes a fixed $q-1$-th rooth of $-\theta$.
The Carlitz exponential admits an inverse on the disk $v_\infty(z) > - \frac{q}{q-1}$ called the Carlitz logarithm, which is given by:
$$\Log_C(z) := \sum_{k\geq 0} (-1)^k \frac{z^{q^k}}{L_k}$$ 
where $L_k := \prod_{i=1}^k (\theta^{q^i} - \theta)$.

Here we are interested in an algebraic independence result for these logarithms obtained by M. Papanikolas \cite{Papanikolas_2007} in 2007:
\begin{thm}{(Papanikolas 2007.)}{\label{Papanikolas}}
   Let $\lambda_1,...,\lambda_n \in \C_\infty$ satisfy $\exp_C(\lambda_i) \in \overline{k}$ for $i \in [\![1,n]\!]$. 
   If the family $( \lambda_i)_{i=1}^n$ is $k$-linearly independent, then it is $\overline{k}$-algebraically independent.
\end{thm}

This result relies on the theory of $t$-motives, especially through the use of the ABP criterion \cite{Anderson_2004} established by G. Anderson, W. Brownawell ad M. Papanikolas in 2004. Our aim here is to provide a new proof of this theorem, using the Mahler's method.

In the general case, to study the algebraic relations among values of functions of arithmetic nature, the commonly employed method is to construct 'good' functions interpolating these values. One then studies the algebraic relations among these functions, which, under certain technical conditions, yield the set of relations among the values under consideration.
Three particular frameworks can be distinguished for this transition to functions:\\
• if our functions are solutions of differential systems (Siegel-Shidlovskii's theorem \cite{Shidlovskii+1989});\\
• if our functions are solutions of $\sigma$-difference systems (ABP theorem \cite{Anderson_2004});\\
• if our functions are solutions of a Mahler system (Nishioka's theorem \cite{Nishioka1990}).

We focus here on the case of Mahler systems. For a natural number $d \geq 2$, we say that a function $f(z) \in \overline{k}[\![z]\!]$ is a $d$-Mahler function if there exists $m\geq 1$ and polynomials $P_0(z),...,P_m(z) \in \overline{k}[z]$ where $P_0(z) \neq 0$ and $P_m(z) \neq 0$ such that
$$P_0(z)f(z) + P_1(z)f(z^d) + P_2(z)f(z^{d^2}) +.... + P_m(z)f(z^{d^m}) = 0.$$
Likewise, a column vector $F(z) = (f_1(z),...,f_n(z))^T$ of functions in $\overline{k}[\![z]\!]$ is said to satisfy a $d$-mahler system if there exists $A(z)$ an invertible matrix with coefficients in $\overline{k}(z)$ such that
$$F(z^d) = A(z) F(z).$$
If $f$ is a $d$-Mahler function, then the vector $F(z) = (f(z),...f(z^{d^{n-1}}))^T$ is a solution of a $d$-mahler system where $A(z)$ is a companion matrix, invertible because $P_0(z) \neq 0$.

The major breakthrough in Mahler’s method came in 1990 with Kumiko Nishioka’s remarkable result in the case of number fields: under certain technical hypotheses, the algebraic relations among the values of Mahler functions at an algebraic point arise solely from the algebraic relations among the functions themselves. A more recent proof of this theorem was given in 2023 by B. Adamczewski and C. Faverjon in \cite{Adamczewski_Faverjon} (cf. \cite{adamczewski2020mahlersmethodvariablesfinite} for Mahler's method in several variables). 
However, these results focus on fields of characteristic zero.

F. Pellarin, in his Bourbaki seminar \cite{seminaire_Bourbaki_Pellarin}, presents the principal transcendence results for the positive characteristic case, among which is the following:
in 2006, Laurent Denis \cite{Denis_2006} established, using a Mahler system, the equivalence between linear and algebraic independence independance of the Carlitz logarithm for elements of $\F_q(\theta^{1/e})$ within its disk of convergence, where $e$ is a positive integer.
His idea was based on a deformation of the denominators $L_k := \prod_{i=1}^k (\theta^{q^i} - \theta)$ into functions $$L_k(z) := \prod_{i=1}^k (z^{q^i} - \theta).$$
The resulting deformation of the logarithm then satisfies a Mahler system of a particular form, for which Laurent Denis established an analogue of Ku. Nishioka’s theorem in \cite{Denis_2000}.

Since then, G. Fernandes gave, in 2018, a general analogue to Ku. Nishioka's theorem for the positive characteristic case in \cite{Fernandes} building on the particular case handled by L. Denis :

\begin{thm}{(Fernandes 2018.)}{\label{Fernandes}} $\text{ }$
Let $\K \subset \overline{k}$ be a finite extension of $k$. Let $n \geq 1$,
$d \geq 2$ be two integers and $f_1(z), . . . , f_n(z) \in \K[\![z]\!]$ be $v_\infty$-adic analytic functions defined in a neighborhood of the origin satisfying the $d$-Mahler system:
$$\begin{pmatrix}
    f_1(z^d)\\ \vdots \\f_n(z^d)
\end{pmatrix} = A(z) \begin{pmatrix}
    f_1(z)\\ \vdots \\f_n(z)
\end{pmatrix}$$
where $A(z) \in \GL_n(\K(z))$.

Let $\alpha \in \overline{k}\backslash\{0\}$ be such that $v_\infty(\alpha) > 0$ and such that for all $k \in \N$, the matrix $A(\alpha^{d^k})$ is well defined and invertible. Then the following equality holds:
$$\degtr_{k}\big(f_1(\alpha),...,f_n(\alpha)\big) = \degtr_{k(z)}\big(f_1(z),...,f_n(z)\big).$$
\end{thm}

We will use a slightly more general version of this theorem, essentially based on the work of D. Adam and L. Denis in \cite{Adam_Denis}, a proof of which will be given in the appendix.

\begin{thm}{\label{Adam_Denis}}
Let $\K \subset \overline{k}$ be a finite extension of $k$. Let $n \geq 1$,
$d \geq 1$ be two integers and $f_1(z), . . . , f_n(z) \in \K[\![z]\!]$ be $v_\infty$-adic analytic functions defined in a neighborhood of the origin satisfying the $q^d$-Mahler system:
$$\begin{pmatrix}
    f_1(z^{q^d})\\ \vdots \\f_n(z^{q^d})
\end{pmatrix} = A(z) \begin{pmatrix}
    f_1(z)\\ \vdots \\f_n(z)
\end{pmatrix}$$
where $A(z) \in \GL_n(\K(H))$, $H= \overline{\F_q}[\![Z]\!] \cap \overline{\F_q(Z)}$.

Let $\alpha \in \overline{k}\backslash\{0\}$ be such that $v_\infty(\alpha) > 0$ and such that for all $k \in \N$, the matrix $A(\alpha^{q^{kd}})$ is well defined and invertible. Then the following equality holds:
$$\degtr_{k}\big(f_1(\alpha),...,f_n(\alpha)\big) = \degtr_{k(z)}\big(f_1(z),...,f_n(z)\big).$$
\end{thm}

Recall that $\Tilde{\theta} \in \C_\infty$ is a $q-1$-th rooth of $-\theta$.
We here propose interpolating functions of the Carlitz logarithm generalizing those constructed by L. Denis.
To this end, let us write $u_i = \exp_C(\lambda_i)$ and assume that $v_\infty(u_i)> -\frac{q}{q-1}$ for $i \in [\![1,n]\!]$. Let $u \in \overline{k}$ be a uniformizer such that $k_\infty(\Tilde{\theta},u_1,...u_n) = \F_{q^d}(\!(u)\!)$ for an integer $d \geq 1$.
We first define the functions $\mathbf{l}(Z) \in \F_{q^d}[\![Z]\!]$ and $\mathbf{l}_i(Z) \in \F_{q^d}[\![Z]\!]$
such that $\mathbf{l}(u) = \frac{1}{\theta}$ and $\mathbf{l}_i(u) = \frac{u_i}{\theta \Tilde{\theta}}$, following the idea of D. Adam et L. Denis in \cite{Adam_Denis}.
We then introduce the following functions:
$$G_i(Z) = G_i(Z,\beta,s) :=  \sum_{k \geq 0} \frac{\mathbf{l}_{i}^{q^k}(Z)}{\pi_k(Z,\beta,s)} \in \F_{q^d}(\theta,\beta)[\![Z]\!],$$
whith $s$ a positive integer, $\beta \in \overline{k}^*$, and where $\pi_k(Z) =\pi_k(Z,\beta,s):= \beta^{-k} \prod_{i=1}^k (1 - \theta \mathbf{l}(Z)^{q^i})^s$;
constructed so that $ \theta\Tilde{\theta} G_i(u,1,1) = \Log_C(u_i)$.

These functions satisfy the functional equation:
$$G_i(Z^{q^{d}}) = \pi_d(Z)( G_i(Z) - M_i(Z))$$
where $M_i(Z) = \sum_{k=0}^{d-1} \frac{\mathbf{l}_i^{q^k}(Z)}{\pi_k(Z)}$.This equation yields, on the one hand, a $q^d$-Mahler system satisfying the conditions of the \color{blue}theorem \ref{Adam_Denis}\color{black}, and on the other hand, the convergence of the sequence $\left(\sum_{j=0}^{k-1} \frac{M_i(Z^{q^{jd}})}{\pi_{jd}(Z)}\right)_{k\geq 1}$ to $G_i(Z)$.

A key step in the proof will be to also consider an interpolating function of the period $\Tilde{\pi}$, closely linked to the denominators $\pi_k(Z)$. Let us then consider the function:
$$G_0(Z) = G_0(Z,s) := \prod_{k=1}^\infty \left(1 - \theta \mathbf{l}(Z)^{q^k}\right)^{-s} = \underset{k\to \infty}{\lim}\pi_k(Z,1,s)^{-1}$$
where $s$ is a positive integer, constructed in such a way that $G_0(u,1) = \frac{1}{\theta\Tilde{\theta}} \Tilde{\pi}$. It also satisfies a functional equation:
$$G_0(Z^{q^d}) = \beta^d \pi_d(Z) G_0(Z)$$
which is similar to the one satisfied by the functions $G_i$ for $i\geq 1$.

Let $\mathbf{B}$ be the purely inseparable closure $\overline{\F_q}(\theta,\beta)$ in $\C_\infty$, we will prove the following result  for these function:

\begin{thm}{\label{1.4}}
    Let $a \in \{0, 1\}$. If the family $(G_i(u))_{i=a}^n$ is $\mathbf{B}$-linearly independent, then it is $\overline{k}$-algebraicly independent.
\end{thm}

In order to ultimately reach M. Papanikolas's theorem, it will be necessary to descend to linear independence over $k$. The method we use allows this descent only under additional restrictive conditions on the parameter $\beta$.

\begin{thm}{\label{1.5}}
Let $a \in \{0, 1\}$ and $\beta \in \bigcup_{h \in \N} \F_{q}(\theta^{1/p^h})$.
If the family $(G_i(u))_{i=1}^n$ is $k$-linearly independent, then it is $\overline{k}$-algebraicly independent.
\end{thm}

We only need the case $\beta = 1$ to conclude on the independence of the logarithms. The presence of $\Tilde{\pi}$ in the case $(\beta,s) = (1,1)$ then allow us to fully obtain Papanikolas's \color{blue}theorem \ref{Papanikolas}\color{black}.\\

\paragraph{\textbf{Acknowledgments:}}
The author would like to warmly thank David Adam for numerous helpful suggestions, Federico Pellarin for constructive comments, and of course Bruno Anglès and Vincent Bosser for their valuable advice and the many stimulating discussions during the preparation of this article.

\newpage
\section{Algebraic Independence of logarithms}

\begin{notations} \text{ }\\
• $k := \F_q(\theta)$;\\
• $k_\infty := \F_q(\!(1/\theta)\!) $;\\
• $\C_\infty$ a completion of an algebraic closure of $k_\infty$;\\
• $v_\infty : \C_\infty \rightarrow \Q \cup \{+\infty\}$ the valuation such that $v_\infty(\theta) = -1$;\\
• $v_Z$ the valuation on $\C_\infty(\!(Z)\!)$ such that $v_Z(Z) = -1$;\\
• $\overline{k}$ the algebraic closure of $k$ contained in $\C_\infty$;\\
• $D_\infty(r) := \{ z \in \C_\infty, v_\infty(z) > r \}$ for $r \in \R$;\\
• $\Tilde{\theta}$ a fixed $q-1$-th rooth of $-\theta$ in $\overline{k}$.
\end{notations}

We recall the definition of the Carlitz Logarithm: 
$$\Log_C(z) := \sum_{k\geq 0} (-1)^k \frac{z^{q^k}}{L_k}$$ 
with $L_k := \prod_{i=1}^k (\theta^{q^i} - \theta)$. Using the computation of the valuation $v_\infty(L_k) = - \frac{q^{k+1}-q}{q-1}$, this logarithm is defined for all elements $z \in D_\infty(-\frac{q}{q-1})$.\\
Next, we define the sequence $\pi_k$ as follows:
$$\pi_k := \prod_{i=1}^k (1-\theta^{1-q^i}) = \theta^\frac{q-q^{k+1}}{q-1} L_k = (-1)^{k} \Tilde{\theta}^{q-q^{k+1}} L_k = (-1)^{k} (\theta\Tilde{\theta})^{1-q^k}L_k.$$
Hence, the Carlitz logarithm can be expressed as follows:
$$\Log_C(z) = \theta\Tilde{\theta}\sum_{k\geq 0} \frac{z^{q^k}}{(\theta\Tilde{\theta})^{q^k} \pi_k}.$$
The point of this representation is that we no longer consider the function $\Log_C$ defined on $D_\infty(-\frac{q}{q-1})$, but the function
$$z \mapsto \sum_{k\geq 0}\frac{z^{q^k}}{\pi_k}$$
defined on $D_\infty(0)$.

\subsection{Construction of the interpolation functions} \text{ }

Let $u_1,...u_n$ be elements of $\overline{k}$ with valuation $> -q/(q-1)$.
Let $K = k[\Tilde{\theta},u_1,...,u_n]$ be a finite extension of $k$. Consider $A_v = \{x \in K, v_\infty(x) \geq 0 \}$, the discrete valuation ring associated with the field $K$ and the valuation $v_\infty$. We note $u$ an uniformizer of this ring. It follows that the element $u$ has strictly positive valuation  and is algebraic over $k$.\\
Let $K_\infty = k_\infty .K \subset \C_\infty$. This is a finite extension of $k_\infty$, and therefore $K_\infty$ is a local field with $u$ as a uniformizer, and whose residue field is a finite extension of $\F_{q}$, that is, of the form $\F_{q^{d}}$ for some integer $d$.
The valuation ring $\O_\infty$ of $K_\infty$ can thus be written $\O_\infty = \F_{q^{d}}[\![u]\!]$. 
Consequently, every element of $K$ with strictly positive valuation can be represented as a power series in the uniformizer $u$ with coefficients in $\F_{q^d}$.

Now, observe that the elements $\frac{1}{\theta}$ and $\frac{1}{\theta\Tilde{\theta}}u_i$ all have strictly positive valuation. Hence, they all belongs to the set $u\F_{q^{d}}[\![u]\!]$.

Let $e := \frac{1}{v_\infty(u)}$ denote the ramification index of $u$ in the extension $K/k$.
Since $v_\infty(\frac{1}{\theta}) = 1$, we obtain the more precise result: 
$$\frac{1}{\theta}\in u^e\F_{q^{d}}[\![u]\!]  \text{ and } \frac{1}{\theta \Tilde{\theta}}u_i\in u\F_{q^{d}}[\![u]\!] .$$
We then define the functions $\mathbf{l}(Z)$ and $\mathbf{l}_i(Z)$ in $\F_{q^d}[\![Z]\!]$ via these expansions:  $$\frac{1}{\theta} = \mathbf{l}(u) \text{ and } \frac{1}{\theta \Tilde{\theta}}u_i = \mathbf{l}_i(u).$$

Using these two functions, we define interpolation function $G_i$ for the Carlitz logarithm of $u_i$ for $i \in [\![1,n]\!]$, by:
$$G_i(Z) = G_i(Z,\beta,s) :=  \sum_{k \geq 0} \frac{\mathbf{l}_{i}^{q^k}(Z)}{\pi_k(Z,\beta,s)}$$
whith $\beta \in \overline{k}^*$, $s$ a positive integer,
and where the denominator is given by: 
$$\pi_k(Z) = \pi_k(Z,\beta,s):= \beta^{-k} \prod_{i=1}^k (1 - \theta \mathbf{l}(Z)^{q^i})^s,$$
so that $ G_i(u,1,1)= \frac{1}{\theta\Tilde{\theta}} \Log_C(u_i)$.\\

Finally let $\Tilde{\pi} = \theta\Tilde{\theta} \prod_{k=1}^\infty \left(1- \theta^{1-q^k}\right)^{-1}$ denote the fundamental period of the Carlitz exponential. We set:

$$G_0(Z) = G_0(Z,s) := \prod_{k=1}^\infty \left(1 - \theta \mathbf{l}(Z)^{q^k}\right)^{-s} = \underset{k\to \infty}{\lim}\pi_k(Z,1,s)^{-1}$$
where $s$ is a positive integer, so that $G_0(u,1) = \frac{1}{\theta\Tilde{\theta}} \Tilde{\pi}$.

\begin{prop}
    The functions $\{G_i(z)\}_{i=0}^n$ are formal power series in z that converge as entire series on $D_\infty(\frac{1}{qe} )$ with coefficients in $\F_{q^{d}}(\theta, \beta)$.
\end{prop}

\begin{proof}
Let us first consider the case $i \in [\![1,n]\!]$.
The functions $\mathbf{l}(z)$ and $\mathbf{l}_{i}(z)$ are, by definition, formal power series. 
Since their coefficients lie in $\F_{q^{d}}$, they define entire series on $D_\infty(0)$.
On the other hand, for each $k \geq 0$, the function 
$\frac{1}{\pi_k(z)} =  \beta^k \prod_{i=1}^k \left( \sum_{m\geq 0}  \theta^m \mathbf{l}(z)^{q^im}\right)^s $ is a formal series with coefficients in $\F_{q^d}(\theta,\beta)$. Using the valuation calculation
$v_\infty(\theta^m \mathbf{l}(z)^{q^im}) = (eq^iv_\infty(z)-1)m$, we see that the function $\frac{1}{\pi_k(z)}$ defines an entire series on $D_\infty(\frac{1}{qe})$. 
consequently, the functions $\frac{\mathbf{l}_{i}^{q^k}(z)}{\pi_k(z)}$ also define entire series on $D_\infty(\frac{1}{qe})$ with coefficients in $\F_{q^d}(\theta, \beta)$.

Consider the valuations with respect to $Z$. We have:\\
• $\forall i \in [\![1,n]\!], v_Z( \mathbf{l}_i(Z)^{q^k}) = q^kv_Z(\mathbf{l}_i(Z)) \geq q^k$ ;\\ 
• $v_Z(\pi_k(Z)) = 0$.\\
It follows that $v_Z\Big(\frac{\mathbf{l}_{i}^{q^k}(Z)}{\pi_k(Z)}\Big) \geq q^k \underset{k \to \infty} \longrightarrow \infty$. 
Hence the series $G_i(Z)$ also defines a formal power series with coefficients in $\F_{q^d}(\theta, \beta)$. Moreover, note that $v_Z(G_i(Z)) > 1$.

Now let $z \in \C_\infty$ be such that $v_\infty(z) > \frac{1}{qe}$. We then have the following estimates:\\
• $\forall i \in [\![1,n]\!], v_\infty(\mathbf{l}_i(z)^{q^k}) \geq v_\infty(z)q^k $;\\
• $v_\infty(\pi_k(z)) = kv_\infty(\beta) + s \sum_{i=1}^k \min(0, v_\infty(\Tilde{\theta \mathbf{l}}(z)^{(q-1)q^i})) \leq  kv_\infty(\beta).$\\
Thus we obtain: 
$$v\Big(\frac{\mathbf{l}_{i}^{q^k}(z)}{\pi_k(z)}\Big) \geq kv_\infty(\beta) + v_\infty(z)q^k > kv_\infty(\beta) + \frac{1}{e}q^{k-1} \underset{k \to \infty} \longrightarrow \infty.$$

It follows that the sequence of functions $\Big( \frac{\mathbf{l}_{i}^{q^k}(z)}{\pi_k(z)} \Big)_{k\in \N}$ is a sequence of entire series on $D_\infty(\frac{1}{qe})$ that converges uniformly to 0 on $D_\infty(\frac{1}{qe})$. By \cite[Proposition 42.2]{Schikhof_1985}, the functions $G_i$ therefore defines an entire series on $D_\infty(\frac{1}{qe})$.

Now, let us study the function $G_0$.    
Using the expansion $\displaystyle(1 - \mathbf{l}(Z)^{q^i} \theta)^{-1} = 1 + \sum_{m=1}^\infty \theta^m \mathbf{l}(Z)^{q^im}$, we can write $G_0(z)$ as:
$$G_0(z) = \prod_{k=1}^\infty \left(1 - \theta \mathbf{l}(z)^{q^k}\right)^{-s} = \prod_{k=1}^\infty (1 + b_k(z))^s,$$ where $b_k(z) := \sum_{m=1}^\infty \theta^m \mathbf{l}(z)^{mq^k}$.
Since $v_Z(b_k(Z)) = eq^k \underset{k \to \infty} \longrightarrow \infty$, the function $G_0(z)$ defines a formal series with coefficients in $\F_{q^d}(\theta)$. 
Moreover, for $z \in D_\infty(\frac{1}{qe})$, we have $v_\infty(b_k(z)) = eq^kv_\infty(z) - 1 > q^{k-1} - 1 \underset{k \to \infty} \longrightarrow \infty$.
Hence the sequence of functions $(b_k(z))_{k=1}^\infty$ converges uniformly to 0 on $D_\infty(\frac{1}{qe})$. By \cite[Proposition 42.2]{Schikhof_1985}, it follows that the products converges to an entire series on $D_\infty(\frac{1}{qe})$.
\end{proof}

\subsection{Functional equation}\text{ }

Let $\displaystyle M_i(Z) = \sum_{k=0}^{d-1} \frac{\mathbf{l}_i^{q^k}(Z)}{\pi_k(Z)}$
be the partial sum of $G_i$ of order $d-1$.
The sequence of functions $(\pi_k(Z))_{k\in \N}$ satisfies the following relation:
\begin{equation}
    \pi_{k}(Z^{q^d}) = \frac{\pi_{k+d}(Z)}{\pi_d(Z)},
    \label{pik}
\end{equation}
Indeed, $\pi_k(Z^{q^d}) = \beta^{-k} \prod_{i=1}^{k}\Big(1 - \theta \mathbf{l}(Z)^{q^{d+i}}\Big)^s = \frac{\beta^d}{\beta^{d+k}} \prod_{i=d +1}^{d + k}\Big( 1 - \theta\mathbf{l}(Z)^{q^i}\Big)^s = \frac{\pi_{d+k}(Z)}{\pi_d(Z)}.$\\
Similarly, we obtain the relation:
\begin{equation}
    \pi_d(Z^{q^{kd}}) = \frac{\pi_{(k+1)d}(Z)}{\pi_{kd}(Z)}.
    \label{pikd}
\end{equation}

Now, using the relation \eqref{pik} inside the definition of $(G_i)_{i=1}^n$, we get:\\
$\pi_d(Z) ( G_i(Z) - M_i(Z))
= \sum_{k \geq d} \frac{\pi_d(Z)}{\pi_k(Z)} \mathbf{l}_{i}^{q^k}(Z)
= \sum_{k \geq 0} \frac{\mathbf{l}_{i}^{q^k}(Z^{q^{d}})}{\pi_k(Z^{q^d})}
= G_i(Z^{q^{d}})$.\\
Thus, we obtain the functional equation: 
\begin{equation}
    G_i(Z^{q^{d}}) = \pi_d(Z)( G_i(Z) - M_i(Z)),
    \label{Gi}
\end{equation} 
and the same for $G_0$:
\begin{equation}
    G_0(Z^{q^d}) = \beta^d \pi_d(Z) G_0(Z).
    \label{G0}
\end{equation}

\begin{lm}{\label{recurrence G_i}}
For all $i \in [\![1,n]\!]$, the functions $G_i$ satisfy the following relation:
$$\frac{G_i(Z^{q^{kd}})}{\pi_{kd}(Z)} = G_i(Z) - \sum_{j=0}^{k-1} \frac{M_i(Z^{q^{jd}})}{\pi_{jd}(Z)}.$$
Moreover, we have the following limit for the valuation $v_Z$:
$$\sum_{j=0}^{k-1} \frac{M_i(Z^{q^{jd}})}{\pi_{jd}(Z)} \underset{k \to \infty}{\longrightarrow} G_i(Z).$$
\end{lm}

\begin{proof}
Let $i \in [\![1,n]\!]$. The first relation can be proved by induction:

$\frac{G_i(Z^{q^{(k+1)d}})}{\pi_{(k+1)d}(Z)}
= \frac{\pi_d(Z^{q^{kd}})(G_i(Z^{q^{kd}})-M_i(Z^{q^{kd}}))}{\pi_{(k+1)d}(Z)}$ (by \eqref{Gi})

$= \frac{ G_i(Z^{q^{kd}}) - M_i(Z^{q^{kd}})}{\pi_{kd}(Z)}$ (by \eqref{pikd})

$= G_i(Z) - \sum_{j=0}^{k-1} \frac{M_i(Z^{q^{jd}})}{\pi_{jd}(Z)} - \frac{M_i(Z^{q^{kd}})}{\pi_{kd}(Z)}$ (by induction)

$= G_i(Z) - \sum_{j=0}^{k} \frac{M_i(Z^{q^{jd}})}{\pi_{jd}(Z)}.$\\ 

Let us now prove the second relation. We have $v_Z(G_i(Z)) \geq 1$, hence $v_Z(G_i(Z^{q^{kd}})) \geq q^{kd}$. Moreover we have $v_Z( \pi_{kd}(Z))= 0$. 
The sequence $\Big( \frac{G_i(Z^{q^{kd}})}{ \pi_{kd}(Z)} \Big)_k$ tends to 0. The recurrence relation previously proved thus provides the desired limit. 
\end{proof}

\subsection{Mahler system}\text{ }

The equations \eqref{Gi} and \eqref{G0} define a $q^d$-Mahler system:\\
Let us define the diagonal matrix $$A(z) = \begin{pmatrix}
    \beta^d\pi_d(z) & 0 & \ldots & 0\\
    0 & \pi_d(z) & & \vdots\\
    \vdots & & \ddots & 0\\
    0 & \ldots & 0 & \pi_d(z)
\end{pmatrix}$$
and the column vecto $$B(z) = \begin{pmatrix}
0\\
\pi_d(z) M_1(z) \\
\vdots \\
\pi_d(z) M_n(z)
\end{pmatrix}.$$
We then have the following system:
$$\begin{pmatrix}
    G_0(z^{q^{d}})\\ G_1(z^{q^{d}})\\ \vdots \\ G_n(z^{q^{d}})
\end{pmatrix} = A(z) \begin{pmatrix}
   G_0(z)\\ G_1(z)\\ \vdots \\ G_n(z)
\end{pmatrix} + B(z).$$ 
By letting $\Tilde{A}$ be the following matrix:

$$\Tilde{A}(z) = \begin{pmatrix}
    \beta^d \pi_d(z) & 0 & \ldots & 0 & 0\\
    0 & \pi_d(z) & & \vdots &\pi_d(z)M_1(z)\\
    \vdots & & \ddots & \vdots & \vdots\\
    \vdots & & & \pi_d(z) & \pi_d(z) M_n(z)\\
    0 & \ldots & \ldots & 0 & 1
\end{pmatrix},$$

we thus obtain the classical form of a Mahler system:

$$\begin{pmatrix}
    G_0(z^{q^{d}})\\ G_1(z^{q^{d}})\\ \vdots \\ G_n(z^{q^{d}})
\end{pmatrix} = \Tilde{A}(z) \begin{pmatrix}
   G_0(z)\\ G_1(z)\\ \vdots \\ G_n(z)
\end{pmatrix}.$$ 
The matrix $\Tilde{A}$ is invertible, whose inverse is:
$$\Tilde{A}^{-1}(z) = \begin{pmatrix}
    \beta^{-d}\pi_d^{-1}(z) & 0 & \ldots & 0 & 0\\
    0 & \pi_d^{-1}(z) & & \vdots & -M_1(z)\\
    \vdots & & \ddots & \vdots & \vdots\\
    \vdots & & & \pi_d^{-1}(z) & - M_n(z)\\
    0 & \ldots & \ldots & 0 & 1
\end{pmatrix}.$$

The series $G_i(Z) \in \F_{q^d}(\theta,\beta)[\![Z]\!]$ are entires for the valuation $v_\infty$ on $D_\infty(\frac{1}{qe})$. Considering the element $u$ defined above, we have:\\
• $0< |u|_\infty  < 1$;\\
• $\pi_d(u^{q^{jd}}) = \beta^{-d} \prod_{k=jd+1}^{(j+1)d} (1 - \theta^{1-q^k})^s \in k(\beta)\backslash\{0\}$;\\
• $M_i(u^{q^{jd}}) = \frac{\pi_{jd}^s}{\beta^{jd}} \sum_{k=jd}^{(j+1)d-1} \frac{\beta^k}{\pi_k^s} \left(\frac{u_i}{\theta\Tilde{\theta}}\right)^{q^k} \in k(\beta)$;\\
Thus, for every integer $j \in \N$, the elements $u^{q^{jd}}$ are neither poles of the entries of $\Tilde{A}(z)$, nor of those of $\Tilde{A}^{-1}(z)$. The matrices $\Tilde{A}(u^{q^{jd}})$ and $\Tilde{A}^{-1}(u^{q^{jd}})$ are therefore well-defined for every integer $j \in \N$.

Moreover, the functions $\mathbf{l}_i(z)$ and $\mathbf{l}(z)$ are algebraic over $k(z)$ by construction, which shows that all the entries of the matrix $\Tilde{A}(z)$ are as well.
We thus obtain from \color{blue}theorem \ref{Adam_Denis} \color{black} the following result:
\begin{prop}{\label{Nishioka G_i}}
    The family $\big( G_i(Z)\big)_{i=0}^n$ is $\overline{k(Z)}$-algebraicly independent if and only if the family $\big( G_i(u)\big)_{i=0}^n$ si $\overline{k}$-algebraicly independent.
\end{prop}

It thus remains to show that the linear independence of the values implies the algebraic independence of the functions.

\subsection{Preliminary constructions}\text{ }

\begin{defn}
Let $\mathbf{B} := \bigcup_{h \in \N} \overline{\F_{q}}(\theta^{1/p^h},\beta^{1/p^h})$ be the purely inseparable closure of $\overline{\F_q}(\theta,\beta)$ in $\C_\infty$, and $\mathbf{U}_Z := \bigcup_{k \in \N} \mathbf{B}(\!(Z^{1/p^k})\!)$. Let $\mathbf{u}_Z$ denote the algebraic closure of $\mathbf{B}(Z)$ in $\mathbf{U}_Z$.
Every element $f \in \mathbf{U}_Z$ admits a unique representation $f = \sum_{i \geq i_0} h_i Z^{i/p^k}$ where $h_i \in \mathbf{B}$, $h_{i_0} \neq 0$. We extend the valuation with respect to $Z$ to $\mathbf{U}_Z$ by $v_Z(f) := i_0/p^k$. 

\end{defn}

Note that these spaces were defined so as to be stable under taking the $p$-th rooth.

\begin{lm}{\label{l,l_i... algebrique}}
The elements $\mathbf{l}(Z)$, $\pi_k(Z)$, $\mathbf{l}_i(Z)$ and $M_i(Z)$ belong to $\mathbf{u}_Z$.
\end{lm}

\begin{proof}
The element $u$ is algebraic over $\F_q(1/\theta)$, so $\mathbf{l}(u) = 1/\theta$ is algebraic over $\F_q(u)$. Because $u$ is transcendental over $\F_q$, we deduce that $\mathbf{l}(Z)$ is algebraic over $\F_q(Z)$, and so is $\pi_k(Z)$, because $\beta \in \mathbf{B}$.
Similarly, $\mathbf{l}_i(Z)$ is algebraic over $\F_q(Z)$, and $M_i(Z)$ is algebraic over $\F_{q^d}(\theta)(Z)$.    
\end{proof}

\begin{lm}{\label{prolongement derivation}}
    The map 
$$\begin{array}{ccccc}
\mathbf{D} & : & \overline{\F_q}(\theta)(\!(Z)\!) & \to & \overline{\F_q}(\theta)(\!(Z)\!)  \\
 & & \sum_{i \geq m} a_i Z^i & \mapsto & \sum_{i \geq m} a_i' Z^i
\end{array}$$
 is a derivation extending the usual derivation on $\overline{\F_q}(\theta)$ and stabilizing the set $\overline{\F_q}(\theta)(\!(Z)\!) \cap \mathbf{u}_Z$.
\end{lm}

\begin{proof}
The map $\mathbf{D}$ is linear. Let $aZ^i$ and $bZ^j$ be two monomials in $\overline{\F_q}(\theta)(\!(Z)\!)$. We have $\mathbf{D}(aZ^ibZ^j) = \mathbf{D}(aZ^i)bZ^j + aZ^i\mathbf{D}(bZ^j)$. It therefore remains to prove that the set is stable.

Let $a \in \overline{\F_q}(\theta)(\!(Z)\!) \cap \mathbf{u}_Z$. We note $P_a(Y) = \sum_{i=0}^n a_iX^i \in \overline{\F_q}(\theta)(Z)[Y]$ its minimal polynomial in $\overline{\F_q}(\theta)(Z)$. Let us also denote $P_a^{(\mathbf{D})}(X) := \sum_{i=0}^n \mathbf{D}(a_i) X^i$.
We then have 
$$0 = \mathbf{D}(P_a(\alpha)) = P_a^{(\mathbf{D})}(a) + \frac{\partial P_a}{\partial X}(a) \mathbf{D}(a).$$
The extension $\overline{\F_q}(\theta)(\!(Z)\!)/ \overline{\F_q}(\theta)(Z)$ being separable, we have $\frac{\partial P_a}{\partial Y}(a) \neq 0$. 
Thus, we get
$$\mathbf{D}(a) = - \frac{P_a^{(\mathbf{D})}(a)}{\frac{\partial P_a}{\partial Y}(a)}.$$
It follows that $\mathbf{D}(a) \in \overline{\F_q}(\theta)(\!(Z)\!) \cap \mathbf{u}_Z$.  
\end{proof}

\begin{lm}{\label{U_z sigma stable}}
Assume that $\beta \in \bigcup_{h \in \N} \overline{\F_{q}}(\theta^{1/p^h})$ and consider the endomorphism $\sigma$ defined as follows:
$$\begin{array}{ccccccc}
\sigma & : & \mathbf{U}_Z & \to & \mathbf{U}_Z\\
 & & x \in \overline{\F_q} & \mapsto & x^q\\
 & & \theta & \mapsto & \theta\\
 & & Z & \mapsto & Z^q.
\end{array}$$
This endomorphism stabilizes $\mathbf{u}_Z$ \color{black}.  
\end{lm}

\begin{proof}
    Let $a \in \mathbf{u}_Z$ and $P_a(Y) \in \mathbf{B}(Z)[Y]$ be an annihilating polynomial of $a$. Because the endomorphism $\sigma$ stabilizes $\mathbf{B}(Z)$, it suffices to apply it to each coefficient of $P_a(Y)$ to get an annihilating polynomial of $\sigma(a)$ with coefficients in $\mathbf{B}(Z)$. \end{proof}

\begin{lm}{\label{lemme annulation eq fonct alg}}
Let $f \in \mathbf{u}_Z$, $P \in \mathbf{B}[\![Z]\!] \cap \mathbf{u}_Z$ and $r$ be a positive integer. We assume that $f$ satisfies the following functional equation:
\begin{equation}
    f(Z^{q^r}) = P(Z)f(Z).
    \label{2.8}
\end{equation}
Assume that $P$ has a pole not belonging to $\overline{\F_q}$, which is a zero of a certain $Q(Z) \in \mathbf{B}[\![Z]\!] \cap \mathbf{u}_Z$, such that for all positive integer $k$:\\
1. $P$ has no poles at the zeros of $Q(Z^{q^{kr}})$;\\
2. $P$ has no zeros at the zeros of $Q(Z^{q^{-kr}})$;\\
3. for all $l \in \N$, the zeros of $Q(Z^{q^{kr}})$ and $Q(Z^{q^{lr}})$ are distincts.\\
Then $f =0$.

\end{lm}

\begin{proof}
From the relation \eqref{2.8}, the pole of $P$ which is a zero of $Q(Z)$ will gives rise either to a zero of $f(Z)$, or to a pole of $f(Z^{q^r})$.\\
In the first case, we will get that $f(Z^{q^r})$ has a common zero with $Q(Z^{q^r})$. This zero cannot be compensated by a pole of $P$ from item 1, hence is also a zero of $f$. It follows by induction that $f$ has a common zero with each $Q(Z^{q^{kr}})$, and therefore has infinitely many from item 3.\\
In the second case, we will get that $f(Z^{q^r})$ has a pole which is a zero of $Q(Z)$, and hence that $f(Z)$ has a pole which is a zero of $Q(Z^{q^{-r}})$. From item 2, this pole cannot be compensated by a zeo of $P$. It follows by induction that $f$ has a pole at a zero of each $Q(Z^{q^{-kr}})$, and therefore has infinitely many from item 3.\\
In both case, since $f$ is algebraic, we obtain that $f=0$.

\end{proof}

\begin{cor}{\label{G_0 transcendante}}
    The function $G_0$ is transcendental over $\mathbf{B}(Z)$.
\end{cor}

\begin{proof}
Using the relation \eqref{G0}, we apply \color{blue}lemma \ref{lemme annulation eq fonct alg} \color{black}to $f(Z) = G_0(Z)^{-1}$ , $r=d$, $P(Z) = \beta^{-d}\pi_d(Z)^{-1}$ and $Q(Z) = \mathbf{l}(Z)^{q^i} - \theta^{-1}$ for any $i \in [\![1,d]\!]$. Because $G_0(Z)^{-1} \neq 0$, $G_0 \notin \mathbf{u}_Z$.
\end{proof}

\begin{lm}{\label{lemme ideal annulateur monogene}}
   Assume that there exists a minimal integer $i_0$ such that $G_{i_0}(Z)$ is algebraic over $k(Z)(\{G_i(Z)\}_{i=0}^{i_0})$. Then there exists a polynomial $P \in \overline{\F_q} (\theta)(Z)[X_0,...,X_{i_0}]$ such that $P(G_0(Z),...,G_{i_0}(Z)) = 0$ and which generates the algebraic relations among the $G_i$ over $\overline{\F_q} (\theta)(Z)$. 
\end{lm}

\begin{proof}
Let $Q$ be a polynomial vanishing at $(G_0(Z),...,G_{i_0}(Z))$. 
The ring\\
$\overline{\F_q} (\theta)(Z)(G_0(Z),...,G_{i_0-1}(Z))[X_{i_0}]$ is principal, hence there exists $P$ a minimal polynomial of $G_{i_0}(Z)$ in this ring dividing $Q(G_0(Z),...,G_{i_0-1}(Z), X_{i_0})$. Multiplying $P$ by powers of the $G_i(Z)$ for $i \in [\![0, i_0-1]\!]$ if necessary, we may assume that $P$ belongs to $\overline{\F_q} (\theta)(Z)[G_0(Z),...,G_{i_0-1}(Z),X_{i_0}]$.
As $i_0$ is minimal, we obtain the isomorphism
$$\overline{\F_q} (\theta)(Z)(G_0(Z),...,G_{i_0-1}(Z)) \sim \overline{\F_q} (\theta)(Z)(X_0,...,X_{i_0-1}).$$
Consequently, there exists a polynomial $\Tilde{P} \in \overline{\F_q} (\theta)(Z)[X_0,...,X_{i_0}]$ dividing $Q$ in the ring $\overline{\F_q} (\theta)(Z)(X_0,...,X_{i_0-1})[X_{i_0}]$ corresponding to $P$ under this isomorphism, which generates all the algebraic relations among the $G_i(Z)$.    
\end{proof}

\subsection{The main result}\text{ }

\begin{lm}{\label{lemme somme algebrique => K-dependance}}
    Let $\K$ be a finite algebraic extension of $k$ and let $\lambda_i$ for $i \in [\![0,n]\!]$ be elements of $\K$, not all zero. Supposoe that the function $\mathcal{G}(z)$ defined by $\mathcal{G}(z) := \sum_{i=0}^n \lambda_i G_i(z)$ is algebraic over $k(z)$. Then the family $\big(G_i(u)\big)_{i=0}^n$ is $\K$-linearly dependent.
\end{lm}

\begin{proof}
Since the function $\mathcal{G}$ is algebraic, it has finitely many poles. By definition of $\mathcal{G}$, these poles are necessarly zeros of the denominator of $G_i$, i.e. zeros of the 
$\pi_k(z)$ for $k \in \N$.
For all positive integer $r$, we set $z_r := \sqrt[q^{rd}]{u}$ such that $\pi_{rd}(z_r) = 0$. Hence, there exists an integer $r_0$ such that $z_{r_0}$ is not a pole of $\mathcal{G}$. Let us decompose $\mathcal{G}$ as follows:

$$\mathcal{G}(z) = \sum_{i=1}^n \lambda_i \sum_{k=0}^{r_0 d -1} \frac{\mathbf{l}_i^{q^k}(z)}{\pi_k(z)} + \frac{g(z) }{\pi_{r_0d}(z)}.$$
where $g$ is the function defined by
$$g(z) := \lambda_0 G_0(z^{q^{r_0d}}) + \sum_{i=1}^n \lambda_i \sum_{k= r_0 d}^{\infty} \frac{\mathbf{l}_i^{q^k}(z) \pi_{r_0 d}(z)}{\pi_k(z)}.$$
This function satisfy:\\
$g(z) = \lambda_0 G_0(z^{q^{r_0d}}) + \sum_{i=1}^n \lambda_i \sum_{k=0}^{\infty} \frac{\mathbf{l}_i^{q^{r_0d+k}}(z) \pi_{r_0 d}(z)}{ \pi_{r_0d+k}(z)}$

$= \lambda_0 G_0(z^{q^{r_0d}}) + \sum_{i=1}^n \lambda_i \sum_{k=0}^{\infty} \frac{\mathbf{l}_i^{q^k}(z^{q^{r_0d}})}{\pi_k(z^{q^{r_0d}})}$

$= \mathcal{G}(z^{q^{r_0d}})$.

The functions $G_i(z)$ being well-defined at $u$, it follows that the function $g(z)$ is well-defined at $z_{r_0}$.
Since $z_{r_0}$ is a zero of $\pi_{r_0d}(Z)$ but not a pole of $\mathcal{G}(Z)$, we get that $z_{r_0}$ is a zero of the function $g(Z)$.

We thus obtain:
$$0 = g(z_{r_0}) = \mathcal{G}(z_{r_0}^{q^{r_0d}}) = \mathcal{G}(u) = \sum_{i=0}^n \lambda_i G_i(u),$$ which yields the desired linear dependence relation.
\end{proof}

\begin{prop}{\label{prop dependance alg => dependance lineaire}}
Assume that the family of formal series $(G_i(Z))_{i=0}^n$ is $k(Z)$-algebraicly dependent. Then there exists coefficients $\eta_i \in \mathbf{B}$, not all zero, such that $\sum_{i=0}^{n} \eta_i G_i(Z)$ is algebraic over $k(Z)$.
\end{prop}

\begin{proof}
Assume that there exists $i_0$ such that $G_{i_0}(Z)$ is algebraic over $k(Z)(\{G_i(Z)\}_{i=0}^{i_0-1})$. Consider then a minimal $i_0$ with this property.
From \color{blue}lemma \ref{lemme ideal annulateur monogene} \color{black}, there exists $P \in \overline{\F_q} (\theta)(Z)[X_0,...,X_{i_0}]$ such that $P(G_0(Z),...,G_{i_0}(Z)) = 0$, and which generates the algebraic relations among the $G_i(Z)$. Under the change of variable $Z \mapsto Z^{q^d}$, we obtain:
$$P(Z^{q^{d}})(G_0(Z^{q^{d}}),...,G_{i_0}(Z^{q^{d}})) = 0.$$
We substitute the functionnal relations \eqref{Gi} and \eqref{G0} into it :
$$P(Z^{q^{d}})\Big( \beta^d \pi_d(Z)G_0(Z),\Big(\pi_d(Z) ( G_i(Z) - M_i(Z))\Big)_{i=1}^{i_0}\Big) = 0.$$
Since $P$ generates the algebraic relation among the $G_i$, there exists $R(Z) \in \mathbf{u}_Z$ such that: 
$$P(Z^{q^{d}})\Big(\beta^d \pi_d(Z) X_0,\Big(\pi_d(Z) ( X_i - M_i(Z))\Big)_{i=1}^{i_0}\Big)  = R(Z). P(Z)(X_0,...,X_{i_0}).$$

We wish to reduce this polynomial relation among the function $G_i$ to degree 1. To this end, we define $\mathcal{S}$ the set of polynomial $Q$ in $\mathbf{u}_Z[X_0,...,X_{i_0}]$ such that there exists an element $R_Q(Z) \in \mathbf{u}_Z$ satisfying: 
\begin{equation}
Q(Z^{q^{d}})\Big(\beta^d \pi_d(Z)X_0,\Big(\pi_d(Z) ( X_i - M_i(Z))\Big)_{i=1}^{i_0}\Big) = R_Q(Z). Q(Z)(X_0,...,X_{i_0}).
\label{S}
\end{equation}

From \color{blue} lemma \ref{l,l_i... algebrique}\color{black}, the polynomial $P$ is a non-constant element of $\mathcal{S}$ not depending solely on $X_0$. Let us then consider $Q$ a polynomial in $\mathcal{S}$ with minimal, non zero total degree in the $(X_i)_{i \neq0 }$. 

Considering $Q$ as a polynomial in $X_0$, we note that each of its coefficients belongs to $\mathcal{S}$. Consider one of its coefficients with nonzero total degree in the $(X_i)_{i \neq0 }$. The minimality of $Q$ implies that the total degree in the $(X_i)_{i \neq0 }$ of the coefficient is the same as that of $Q$. Hence, we may replace $Q$ by this coefficient and work with a polynomial independent of $X_0$.

Derivating \eqref{S} with respect to $X_i$ for $i \in [\![1,i_0]\!]$, we get:
$$\pi_d(Z)\frac{\partial Q}{\partial X_i}(Z^{q^{d}})\Big( \beta^d \pi_d(Z) X_0,\Big(\pi_d(Z) ( X_i - M_i(Z))\Big)_{i=1}^{i_0}\Big)$$
$$= R_Q(Z) \frac{\partial Q}{\partial X_i}(Z)(X_0,...,X_{i_0}).$$

Hence, the set $\mathcal{S}$ is stable under derivations with respect to the $X_i$. The minimalty of $Q$ implies that for all $i \in [\![1,i_0]\!]$, $\frac{\partial Q}{\partial X_i}$ is constant. We deduce the following expression for $Q$:
$$Q(Z)(X_0,...X_{i_0}) = \sum_{i=1}^{i_0} \eta_i(Z) X_i + \Tilde{Q}(Z)(X_1^p,...,X_{i_0}^p)$$
where $\eta_i(Z)$ are elements in $\mathbf{u}_Z$ and $\Tilde{Q}$ is a polynomial in $\mathbf{u}_Z[X_1,...X_{i_0}]$, with constant term denoted by $\nu(Z) \in \mathbf{u}_Z$, and a non-constant term of maximal degree denoted by $\delta \prod_{i=1}^{i_0} X_i^{j_i}$, $\delta \in \mathbf{u}_Z$.

If all the $\eta_i$ were zero, then $\sqrt[p]{Q}$ would be a polynomial in $\mathcal{S}$ of strictly smaller degree than that of $Q$. Hence, we may assume that  one of the $\eta_i$ is nonzero. Assume $\eta_1 = 1$, up to reordering and multiplying $Q$ by $\eta_1^{-1}$.

Substituting this new expression for $Q$ into equation \eqref{S}, we obtain for all $i \in [\![1,i_0]\!]$:
$$\eta_i(Z^{q^{d}})\pi_d(Z) = R_Q(Z) \eta_i(Z),$$
$$\delta(Z^{q^{d}}) \pi_d(Z)^{p(j_1+...+j_{i_0})} = R_Q(Z) \delta(Z).$$

This yields the following three points:\\
• $\pi_d(Z) = R_Q(Z)$ (avec le cas $i=1$);\\
• $\eta_i(Z^{q^{d}}) = \eta_i(Z)$;\\
• $\delta(Z^{q^{d}}) = \pi_d(Z)^{1-p(j_0+...+j_{i_0} )}\delta(Z)$.\\
The second point thus shows that the $(\eta_i)_{i \in [\![1,i_0]\!]}$ belongs to the set $\mathbf{B}$.

Concerning the third point, let us recall the expression of $\pi_d(Z)^{-1}$:

$$\frac{1}{\pi_d(Z)} = \beta^d\prod_{j=1}^d \left( 1 - \theta \mathbf{l}^{q^j}(Z)\right)^{-s}.$$
- Its zeros are those of $\frac{1}{\mathbf{l}(Z)}$;\\
- Its poles are the zeros of $\mathbf{l}(Z) - \theta^{-q^{-j}}$ for $j \in [\![1,d]\!]$.\\
From \color{blue} lemma \ref{lemme annulation eq fonct alg} \color{black}, choosing $Q(Z) = \mathbf{l}(Z) - \theta^{-q^{-j}}$ and for any $j \in [\![1,d]\!]$, we obtain from the third point that $\delta = 0$.

The expression fo $Q$ thus becomes: 
$$Q(Z)(X_0,...X_{i_0}) = \nu(Z) + \sum_{i=1}^{i_0} \eta_i X_i.$$
We can then substitute this expression back into \eqref{S}:

\begin{equation}
\nu(Z^{q^{d}}) -  \pi_d(Z) \sum_{i=1}^{i_0} \eta_i M_i(Z) =  \pi_d(Z) \nu(Z).
\end{equation}

Observe the similarity between this equation and \eqref{Gi}. Using an induction argument indentical to that in the proof of \color{blue} lemma \ref{recurrence G_i}\color{black}, we obtain for all positive integer $k$ that: 
\begin{equation}
\frac{\nu(Z^{q^{kd}})}{\pi_{kd}(Z)} = \nu(Z) + \sum_{j=0}^{k-1} \sum_{i=1}^{i_0} \eta_i \frac{M_i(Z^{q^{jd}})}{\pi_{jd}(Z)}.  
\label{eq8}
\end{equation}

From \color{blue} lemma \ref{recurrence G_i}\color{black}, we obtain a relation very similar to \eqref{eq8}:
\begin{equation}
\frac{ \sum_{i=1}^{i_0} \eta_i G_i(Z^{q^{kd}})}{ \pi_{kd}(Z)} = \sum_{i=1}^{i_0} \eta_i G_i(Z) - \sum_{j=0}^{k-1} \sum_{i=1}^{i_0} \eta_i \frac{M_i(Z^{q^{jd}})}{\pi_{jd}(Z)}.
\label{eq9}
\end{equation}

Let us work in the field $\mathbf{U}_Z$. According to \color{blue}lemma \ref{recurrence G_i}\color{black}, we have the convergence:
$$ \sum_{j=0}^{k-1} \sum_{i=1}^{i_0} \eta_i \frac{M_i(Z^{q^{jd}})}{ \pi_{jd}(Z)} \underset{k \to \infty}{\longrightarrow} \sum_{i=1}^{i_0} \eta_iG_i(Z).$$
We deduce from \eqref{eq8} that the sequence $\Big( \frac{\nu(Z^{q^{kd}})}{\pi_{kd}(Z)} - \nu(Z) \Big )_k$ also converges to $\sum_{i=1}^{i_0} \eta_iG_i(Z)$.
This convergence implies that the sequence $\Big( \frac{\nu(Z^{q^{kd}})}{\pi_{kd}(Z)}\Big )_k$ converges, and hence that $v_Z(\nu(Z)) \geq 0$, since $v_z(\pi_{kd}(Z)) = 0$.
If the inequality is strict, the sequence converges to $0$ and $\nu(Z) = \sum_{i=1}^{i_0} \eta_i G_i(Z)$.\\
If equality holds, the sequence $\left(\nu(Z^{q^{kd}})\right)_k$ converges to the constant term of $\nu$, which terefore lies in $\mathbf{B}$. 
Since the sequence $\Big( \frac{\nu(Z^{q^{kd}})}{\pi_{kd}(Z)}\Big )_k$ converge, the case $v_Z(\nu(Z)) = 0$ forces the sequence $\left( \beta^{kd} \right)_k$ to converge in $\mathbf{U}_Z$, and hence $\beta$ must have order dividing $d$.
we then have the following convergence:
$$\frac{\nu(Z^{q^{kd}})}{\pi_{kd}(Z)} 
= \nu(Z^{q^{kd}}) \beta^{kd} \prod_{i=1}^{kd} (1 -\theta \mathbf{l}(Z)^{q^i})^{-s}
\underset{k \to \infty}\longrightarrow C \prod_{i=1}^{\infty} (1-\theta \mathbf{l}(Z)^{q^i})^{-s} = C G_0(Z)$$
where $C = \nu(0) \in \mathbf{B}$. It follows that $\nu(Z) = C G_0(Z) - \sum_{i=1}^{i_0} \eta_i G_i(Z)$.

In both case, by taking either $\eta_0 = 0$, or $\eta_0 = -C$, and setting $\eta_i =0$ for all $i>i_0$, we can now conclude that 
$$\sum_{i=0}^{n} \eta_i G_i(Z) \in \mathbf{u}_Z.$$
\end{proof}

Combining \color{blue}proposition \ref{prop dependance alg => dependance lineaire} \color{black} with \color{blue}lemma \ref{lemme somme algebrique => K-dependance} \color{black} and \color{blue}proposition \ref{Nishioka G_i}\color{black}, we obtain:

\begin{cor}{\label{cor B lin indep => k-alg indep}}
    If the family $(G_i(u))_{i=0}^n$ is $\mathbf{B}$-linearly independent, then it is $\overline{k}$-algebraicly independent.
\end{cor}

Finally, we can eliminate the dependence on the function $G_0$, yielding the following result:

\begin{thm}[Theorem \ref{1.4}]{\label{thm B lin indep => k-alg indep}}
    Let $a \in \{0, 1\}$. If the family $(G_i(u))_{i=a}^n$ is $\mathbf{B}$-linearly independent, then it is $\overline{k}$-algebraicly independent.
\end{thm}

\begin{proof}
Assume that the family $(G_i(u))_{i=1}^n$ is $\mathbf{B}$-linearly independent.

From \color{blue} corollary \ref{cor B lin indep => k-alg indep}\color{black}, if the family $(G_i(u))_{i=0}^n$ is $\mathbf{B}$-linearly independent, then it is $\overline{k}$-algebraicly independent, thus the same holds for the family $(G_i(u))_{i=1}^n$. 

Assume now that the family $(G_i(u))_{i=0}^n$ $\mathbf{B}$-linearly dependent. 
Since the family $(G_i(u))_{i=1}^n$ is $\mathbf{B}$-linearly independent, there exists a linear combination $\lambda_0 G_0(u) + \sum_{i=1}^n \lambda_i G_i(u) = 0$ where $\lambda_0 \neq 0$, such that there exists $i_0 \in [\![1,n]\!]$ verifying $\lambda_{i_0} \neq 0$. 
Up to reordering, we can assume that $i_0 = n$. The family $(G_i(u))_{i=0}^{n-1}$ is thus $\mathbf{B}$-linearly independent.
It follows from \color{blue} corollary \ref{cor B lin indep => k-alg indep} \color{black} that the family $(G_i(u))_{i=0}^{n-1}$ is $\overline{k}$-algebraicly independent, and so is the family 
$$\left(G_1(u),..., G_{n-1}(u), \frac{-1}{\lambda_n}(\lambda_0 G_0(u) + \sum_{i=1}^{n-1} \lambda_i G_i(u))\right).$$ In other words, we find that the family $(G_i(u))_{i=1}^n$ is $\overline{k}$-algebraicly independent.
\end{proof}

It then remains to consider the possibility of descending the $\mathbf{B}$-linear independence to a $k$-linear independence. This would correspond to an analogue of Baker’s theorem for function fields, attainable for certain particular values of $\beta$.

\begin{thm}[Theorem \ref{1.5}]{\label{thm k-lin indep => k-alg indep}}
Let $a \in \{0, 1\}$ and $\beta \in \bigcup_{h \in \N} \F_{q}(\theta^{1/p^h})$.
If the family $(G_i(u))_{i=1}^n$ is $k$-linearly independent, then it is $\overline{k}$-algebraicly independent.
\end{thm}

\begin{proof}
It suffices to show that under the assumptions, a $k$-linear independence implies a $\mathbf{B}$-linear independence. Indeed, it will then suffice to combine this result with \color{blue} proposition \ref{prop dependance alg => dependance lineaire}\color{black}, \color{blue} lemma \ref{lemme somme algebrique => K-dependance}\color{black}, and \color{blue}theorem \ref{Nishioka G_i} \color{black}to obtain the analogue of \color{blue}corollary \ref{cor B lin indep => k-alg indep}\color{black}. We conclude in the same manner as the proof of \color{blue} theorem \ref{thm B lin indep => k-alg indep}\color{black}.

Assume, then, that there exist elements $\eta_0,...,\eta_n$ in $\mathbf{B}$ such that:
$$\sum_{i=0}^{i_0} \eta_i G_i(Z) \in \mathbf{u}_Z.$$
Since the $\eta_i$ lie in $\mathbf{B}$, by assumption on $\beta$ there exists an integer $j$ such that the $\eta_i^{p^j}$ lie in $\overline{\F_q}(\theta)$. Up to the multiplication by a monic polynomial in $\overline{\F_q}[\theta]$, we thus obtain polynomials $\mu_i \in \overline{\F_q}[\theta]$, not all zero, such that: 
$$\sum_{i=0}^{i_0} \mu_i G_i^{p^j}(Z) \in \mathbf{u}_Z.$$

Let $j_0$ be the smallest integer $j$ such that there exists polynomial $\mu_{i} \in \overline{\F_q}[\theta]$, not all zero, satisfying
$$\sum_{i=0}^{i_0}  \mu_i G_i^{p^{j_0}}(Z) \in \mathbf{u}_Z.$$

We define a partial order on $(\overline{\F_q}[\theta])^{i_0+1}$ by:
$$(P_i)_{i=0}^{i_0} \leq (Q_i)_{i=0}^{i_0}
\iff 
\forall i \in [\![0,i_0]\!], \deg(P_i) \leq \deg(Q_i).$$
Let $(\kappa_i)_{i=0}^{i_0}$ be a minimal family for this order satisfying
$$\mathcal{G}(Z) := \sum_{i=0}^{i_0} \kappa_i G_i^{p^{j_0}}(Z) \in \mathbf{u}_Z.$$

Since $G_0(Z) \not \in \mathbf{u}_Z$ from \color{blue}corollary \ref{G_0 transcendante}\color{black}, if $\kappa_0 \neq 0$, then there exists a second element $m \in [\![1,i_0]\!]$ such that $\kappa_{m}$ is nonzero.
Mltiplying $\mathcal{G}$ by an element of $\overline{\F_q}$ if necessary, we may always assume that there exists $m \in [\![1,i_0]\!]$ such that $\kappa_m$ is monic and nonzero.

By \color{blue}lemma \ref{prolongement derivation}\color{black}, if $j_0 \neq 0$, $\displaystyle \mathbf{D}(\mathcal{G}(Z)) = \sum_{i=0}^{i_0} \frac{\partial}{\partial \theta}( \kappa_i) G_i^{p^{j_0}}(Z) \in \mathbf{u}_Z$. The family $(\kappa_i)_{i=0}^{i_0}$ being assumed minimal, all the $\frac{\partial}{\partial \theta}( \kappa_i)$ are necessarly zero. Consequently, each $\kappa_i$ lies in $\overline{\F_q}[\theta^p]$. This means that we can estract a $p$-th rooth of each coefficient that still belongs to $\overline{\F_q}[\theta]$. But then the minimalty of $j_0$ is contradicted. It follows that $j_0 = 0$.

From \color{blue} lemma \ref{U_z sigma stable}\color{black}, the endomorphism $\sigma$ stabilizes $\mathbf{u}_Z$, thus $\sigma(\mathcal{G}(Z)) \in \mathbf{u}_Z$. By assumption, it also fixes the element $\beta$.
We then calculate its expression in terms of $\mathcal{G}(Z)$.
For $i \in [\![1,i_0]\!]$, we thus have:
$$\sigma(G_i(Z)) = \sum_{k \geq 0} \pi_1(Z) \left(\frac{\beta}{\sigma(\beta)}\right)^k \frac{\mathbf{l}_i^{q^{k+1}}(Z)}{ \pi_{k+1}(Z)} = \pi_1(Z)\big( G_i(Z) - \mathbf{l}_i(Z) \big);$$
$$\sigma(G_0(Z)) = \prod_{i=1}^\infty \left( 1 - \mathbf{l}(Z)^{q^{i+1}} \theta \right)^{-s} =  \pi_1(Z) G_0(Z).$$
We thus obtain: 
$$\sigma(\mathcal{G}(Z)) = \pi_1(Z) \left( \sigma(\kappa_0) G_0(Z) + \sum_{i=1}^{i_0} \sigma(\kappa_i)  \big(G_i(Z) - \mathbf{l}_i(Z)\big) \right) \in \mathbf{u}_Z.$$
It is known by \color{blue}lemmas \ref{l,l_i... algebrique} \color{black} and \color{blue}\ref{U_z sigma stable} \color{black} that $\sigma(\kappa_i)$, $ \pi_1(Z)$, $\mathbf{l}_i(Z)$ et $\mathbf{l}(Z)$ belongs to $\mathbf{u}_Z$, hence 
$$\sum_{i=0}^{i_0} \sigma(\kappa_i) G_i(Z) \in \mathbf{u}_Z.$$
Consequently,
$$\sum_{i=0}^{i_0} \big( \kappa_i - \sigma(\kappa_i) \big) G_i(Z) = \sum_{i=0}^{i_0} \sum_{k=0}^{n_\alpha-1} \big( \kappa_{i,k} - \sigma(\kappa_{i,k}) \big) \alpha^k G_i(Z)  \in \mathbf{u}_Z.$$

We have for all $i \in [\![0,i_0]\!]$ that $\deg\big( \kappa_i - \sigma(\kappa_i) \big) \leq \deg(\kappa_i)$. Moreover, since $\kappa_m$ is monic, $\deg\big( \kappa_m - \sigma(\kappa_m) \big) < \deg(\kappa_m)$. This implies from the minimality of the family $(\kappa_i)_{i=0}^{i_0}$ that $\sigma(\kappa_i) = \kappa_i$ for all $i\in [\![1,i_0]\!]$. In other words, we get that all the $\kappa_i$ lies in $\F_q[\theta]$, which gives the desired $k$-linear dependence relation.
\end{proof}

\subsection{Application}\text{ }

Evaluating the functions $G_i$ at $(Z,\beta,s) = (u,1,1)$ yields the Carlitz logarithm:
$$G_i(u,1,1) = \frac{1}{\theta \Tilde{\theta}} \Log_C(u_i).$$
We deduce then from \color{blue} theorem \ref{thm k-lin indep => k-alg indep} \color{black} the following result:

\begin{prop}\label{application log carlitz}
   Let $u_1,...,u_n \in \overline{k}$ be $n$ algebraic elements in $D_\infty(\frac{-q}{q-1})$ the disc of convergence of the Carlitz logarithm. 
   If the family $\big( \Tilde{\pi},\Log_C(u_1),...,$ $\Log_C(u_n)\big)$ is $k$-linearly independent, then it is $\overline{k}$-algebraicly independent.
\end{prop}

This proposition is currently less general than Papanikolas's \color{blue}theorem \ref{Papanikolas} \color{black}. To address this, let $\lambda_1,...,\lambda_n$ be $n$ elements in $\C_\infty$ such that for all $i \in [\![1,n]\!]$ $\beta_i := \exp_C(\lambda_i) \in \overline{k}$.

If  $v_\infty(\beta_i) > -\frac{q}{q-1}$, we set $u_i := \beta_i$, and we have $\lambda_i = \Log_C(u_i)$.

Otherwise, we have $v_\infty(\beta_i) \leq -\frac{q}{q-1}$.
Recall that the Carlitz exponential satisfies the relation
$\exp_C(\theta z) = \theta \exp_C(z) + \exp_C(z)^q.$
We then consider $\alpha_i$ a solution of the equation $$\beta_i = \theta \alpha_i + \alpha_i^q.$$
A computation of the valuation shows that we must have $v_\infty(\alpha_i) \leq -\frac{1}{q-1}$. If this inequality is strict, we obtain that $v_\infty(\beta_i) = q v_\infty(\alpha_i)$. Otherwise we obtain that $v_\infty(\beta_i) \geq - \frac{q}{q-1}$, therefore $v_\infty(\beta_i) = - \frac{q}{q-1}$.
In any case, it follows that 
$$v_\infty(\alpha_i) = \frac{v_\infty(\beta_i)}{q} > v_\infty(\beta_i).$$
Let now $\gamma_i$ be such that $\exp_C(\gamma_i) = \alpha_i$. We thus have:
$$\exp_C(\theta \gamma_i) = \beta_i = \exp_C(\lambda_i).$$
If $v_\infty(\alpha_i) > -\frac{q}{q-1}$, we set $u_i = \alpha_i$.

Therefore, for all $i \in [\![1,n]\!]$, there exists $n_i \in \N$ and $u_i \in  D_\infty(-\frac{q}{q-1})$ such that 
$$\beta_i = \exp_C(\theta^{n_i} \Log_C(u_i)).$$
Consequently, for each $i \in [\![1,n]\!]$ there are elements $f_i \in \F_q[\theta]$ such that
$$\lambda_i = \theta^{n_i} \Log_C(u_i) + f_1 \Tilde{\pi}.$$

The linear independence over $k$ of the family $(\Tilde{\pi}, \lambda_1,...,\lambda_n)$ is thus equivalent to that of the family $(\Tilde{\pi}, \Log_C(u_1),...,\Log_C(u_n))$, and the same holds for the algebraic independence over $\overline{k}$.

We deduce from \color{blue}proposition \ref{application log carlitz} \color{black} that the linear independence of the family $(\Tilde{\pi}, \lambda_1,...,\lambda_n)$ implies its algebraic independence. Finally, it remains to eliminate the dependence on $\Tilde{\pi}$, which can be done by a suitable adaptation of the proof of \color{blue} theorem \ref{thm B lin indep => k-alg indep}\color{black}.

Assume that the family $(\lambda_1,...,\lambda_n)$ is $k$-linearly independent, but that the family $(\Tilde{\pi}, \lambda_1,...,\lambda_n)$ is $k$-linearly dependent.
There exists then a linear dependence relation $\mu_0 \Tilde{\pi} + \sum_{i=1}^n \mu_i \lambda_i = 0$ where $\mu_0 \neq 0$ such that there exists $i_0 \in [\![1,n]\!]$ verifying $\mu_{i_0} \neq 0$. 
After possibly reordering, we may assume that $i_0 = n$. The family $(\Tilde{\pi}, \lambda_1,...,\lambda_{n-1})$ is thus $k$-linearly independent.
It follows from \color{blue} proposition \ref{application log carlitz} \color{black} that the family $(\Tilde{\pi}, \lambda_1,...,\lambda_{n-1})$ is $\overline{k}$-algebraicly independent, and one may replace $\Tilde{\pi}$ by $\lambda_n$ thanks to the previously given linear dependence relation.
We thus finally obtain Papanikolas's theorem:

\begin{thm}[Theorem 1.1]
   Let $\lambda_1,...,\lambda_n \in \C_\infty$ satisfy $\exp_C(\lambda_i) \in \overline{k}$ for $i \in [\![1,n]\!]$. 
   If the family $( \lambda_i)_{i=1}^n$ is $k$-linearly independent, then it is $\overline{k}$-algebraically independent.
\end{thm}

\section{Appendix: Proof of theorem \ref{Adam_Denis}}
The statement of the theorem \ref{Adam_Denis} is a slight improvement on the work of G. Fernandes \cite{Fernandes}, D. Adam and L. Denis \cite{Adam_Denis}. For the reader’s convenience, we provide a detailed proof below.
In fact, we prove the following result, from which Theorem \ref{Adam_Denis} follows as a corollary.

\begin{thm}{\label{Adam_Denis2}}
Let $\K \subset \overline{k}$ be a finite extension of $k$. Let $n \geq 1$,
$d \geq 2$ be two integers and $f_1(z), . . . , f_n(z) \in \K[\![z]\!]$ be $v_\infty$-adic analytic functions defined in a neighborhood of the origin satisfying the $d$-Mahler system:
$$\begin{pmatrix}
    f_1(z^d)\\ \vdots \\f_n(z^d)
\end{pmatrix} = A(z) \begin{pmatrix}
    f_1(z)\\ \vdots \\f_n(z)
\end{pmatrix}$$
where $A(z) \in \GL_n(\K(\!(z)\!))$ whose entries are algebraic over $k(z)$.

Let $\alpha \in \overline{k}\backslash\{0\}$ be such that $v_\infty(\alpha) > 0$ and such that for all $k \in \N$, the matrix $A(\alpha^{d^k})$ is well defined, invertible and all its entries lie in $\K$. 
Then the following equality holds:
$$\degtr_{k}\big(f_1(\alpha),...,f_n(\alpha)\big) = \degtr_{k(z)}\big(f_1(z),...,f_n(z)\big).$$
\end{thm}

\begin{proof}[Proof of theorem \ref{Adam_Denis2} $\implies$ theorem \ref{Adam_Denis}]
Let $H= \overline{\F_q}[\![Z]\!] \cap \overline{\F_q(Z)}$. We first have the following inclusion:
$$\K(H) \subset \overline{k}[\![Z]\!] \cap \overline{k(Z)}.$$
Now let $h(Z) \in H$. Since $h(Z)$ is algebraic over $\F_q(Z)$, there exists an integer $l$ such that $h(Z) \in \F_{q^l}[\![Z]\!]$. For all $k \in \N$, write $kd = \lambda l + \mu$, $\mu <l$. Thus, we have:
$$h(Z^{q^{kd}}) = h(Z^{q^{\lambda l + \mu}}) = h(Z^{q^{\mu}})^{q^{\lambda l}}.$$
Therefore, there exists $\L$ a finite extension of $\K$ such that for all $k \in \N$, $h(\alpha^{q^{kd}}) \in \L$.
Thus, after possibly replacing $\K$ by a finite extension, all the assumptions of theorem $\ref{Adam_Denis2}$ hold.
\end{proof}

This proof of theorem \ref{Adam_Denis2} essentially follows the method used by G. Fernandes in \cite{Fernandes} and adaptated by D. Adam and L. Denis in \cite{Adam_Denis}. We will need the following notions:
\begin{defn}
    Let $F$ be a family of elements in $\overline{k}$. We call a denominator of $F$ any nonzero polynomial $D \in \F_q[\theta]$ such that for all $x \in F$, the element $Dx$ is integral over $k$.
\end{defn}

\begin{defn}
    Let $\alpha \in \overline{k}$. We call house of $\alpha$ the element 
    $$\house{\alpha} = \max\{ -v_\infty\big(\sigma(\alpha)\big), \sigma \in \Hom_k(k(\alpha), \C_\infty) \}$$
\end{defn}

\begin{prop}
    Let $x,y \in \overline{K}$ and $n \in \N$. The following properties hold:\\
    • $\house{x+y}\leq \max\{\house{x}, \house{y}\}$;\\
    • $\house{x.y} \leq \house{x} + \house{y}$;\\
    • $\house{x^n} = n.\house{x}$.
\end{prop}

\begin{defn}
Let $\K \subset \overline{k}$ be a finite field extension of $k$ and $a\in k$. We define the  absolute logarithmic Weil height of a by:
$$h(a) = \frac{1}{[\K:k]}\sum_{\omega \in \mathcal{M}_k}d_w \log(\max\{1,|a|_\omega\}).$$
Given a family $\alpha_1,...,\alpha_n \in \overline{k}$, we set:
$$h(\alpha_1,...,\alpha_n)=  \frac{1}{[\K:k]}\sum_{\omega \in \mathcal{M}_k}d_w
\log(\max\{1,|\alpha_1|_\omega,...,|\alpha_n|_\omega \}).$$
Given a polynomial $P \in \K[X_0,...,X_n]$, we define $h(P)$ by the height of the family of its coefficients.\\
\end{defn}

\begin{prop}{\label{inegalite hauteur maison denominateur}}
Let $\alpha_1,...,\alpha_n$ be elements in $\overline{k}$, and let $D$ be a denominator of these elements. The following inequality holds:
$$h(\alpha_1,...,\alpha_n)\leq \max\{0,\house{\alpha_1},...,\house{\alpha_n}\} + \deg_\theta(D).$$
\end{prop}

\subsection{An algebraic independence criterion}\text{ }

The proof of the theorem relies on an algebraic independence criterion due to P. Philippon \cite[Theorem 2]{Philippon}\color{black} :

\begin{thm}{(Philippon, cf. \cite{Philippon},\cite{Fernandes})}{\label{philippon}} $\text{ }$

Let $\omega = (1, \omega_1,...,\omega_n) \in \C_\infty^{n+1}$ and $s \in [\![0,n]\!]$. Assume that for all constant $c > 0$, there exists four non-decreasing real sequences $(\delta_t)_{t \in \N}$, $(\sigma_t)_{t \in \N}$, $(\epsilon_t)_{t \in \N}$ and $(\rho_t)_{t \in \N}$ with values greater than 1 such that:\\
(1) $\delta_t \leq \sigma_t$;\\
(2) $\epsilon_t \leq \rho_{t+1}$;\\
(3) $\delta_t + \sigma_t \underset{t \to \infty}{\longrightarrow} \infty$;\\
(4) The sequence $\left( \frac{\epsilon_t}{(\delta_t+ \sigma_t) \delta_t^s} \right)_{t \in \N}$ is non-decreasing;\\
(5) $\frac{\epsilon_t^{s+1}}{\delta_t^{s-1}(\epsilon_{t+1}^s + \rho_{t+1}^s)} \geq c (\delta_t + \sigma_t) $.\\
Assume furthermore that for all non negative integer $t$, there exists an homogeneous polynomial $P_t \in \K[X_0,...,X_n]$ satisfying the following conditions:\\
(6) $\deg(P_t) \leq \delta_t$;\\
(7) $h(P_t) \leq \sigma_t$ ;\\
(8) $-\rho_t \leq \log|P_t(\omega)| \leq - \epsilon_t$.\\
Then we have: 
$$degtr_k\{\omega_1,...,\omega_n\} \geq s.$$
\end{thm}

\begin{cor}{(Fernandes, cf. \cite{Fernandes} et \cite{Adam_Denis})}{\label{cor philippon}} $\text{ }$

Let $d$ be an integer $\geq 2$, $\omega = (1, \omega_1,...,\omega_n) \in \C_\infty^{n+1}$, $s \in [\![0,n]\!]$ and $c_1,c_2,...,c_6$ be positive constants independent of $N$ and $t$ with $c_5 \geq c_6$. 
Let $(n_N)_{N \in \N}$ also be a sequence of integers such that:\\
• $n_N \geq c_2 N^{s+1}$.\\
Assume that for all integers $N \geq c_1$ and $t \geq 0$, there exists a positive number $c_0(N)$ depending only on $N$ and a homogeneous polynomial $R_{N,t} \in \K[X_0,...,X_n]$ satisfying:\\ 
• $\deg(R_{N,t})\leq c_3N$;\\
• $h(R_{N,t}) \leq c_4 d^{t+c_0(N)} N$;\\
• $-c_5d^{t+c_0(N)} n_N \leq \log|R_{N,t}(\omega)| \leq -c_6 d^{t+c_0(N)} n_N$.\\
Then we have: 
$$degtr_k\{\omega_1,...,\omega_n\} \geq s.$$
\end{cor}

\begin{proof}
Assume that all the assumption of the corollary are satisfied. Let us consider a positive constant $c$ and an integer $N \in \N$ chosen sufficiently large. We set:\\
• $\delta_t = c_3N$;\\
• $\sigma_t = c_4 d^{t+c_0(N)} N$;\\
• $\epsilon_t = c_6 d^{t+c_0(N)} n_N$;\\
• $\rho_t = c_5 d^{t+c_0(N)} n_N$;\\
Our task is therefore to prove that these sequences fulfill the conditions of \color{blue}theorem \ref{philippon}\color{black}.
The verification of conditions (6), (7) and (8) is immediate. Since $c_5 \geq c_6$, we have that $\rho_t \geq \epsilon_t$. It follows that condition (2) holds. 
Conditions (1) and (3) also follow naturally, up to increasing $c_0(N)$ for (1).
Let us verify condition (4):
$$u_t := \frac{\epsilon_t}{(\delta_t+ \sigma_t) \delta_t^s} = \frac{c_0 d^{t+c_0(N)} n_N}{(c_3 + c_4d^{t+c_0(N)})c_3^sN^{s+1}};$$

$$\frac{u_{t+1}}{u_t} = \frac{d(c_3 + c_4d^{t+c_0(N)})}{c_3 + c_4 d^{t+1+c_0(N)}} = 1 + \frac{c_3(d-1)}{c_3 + c_4 d^{t+1+c_0(N)}} > 1.$$
Hence is this sequence non-decreasing. We now turn to condition (5):

$$\frac{\epsilon_t^{s+1}}{\delta_t^{s-1}(\epsilon_{t+1}^s + \rho_{t+1}^s)} = \frac{c_6^{s+1} }{c_3^{s-1}d^s(c_6^s + c_5^s)} \frac{d^{t+c_0(N)}n_N}{N^{s-1}}.$$
Now, the sequence $(n_N)_{N \in \N}$ is defined so as to satisfy $n_N \geq c_2 N^{s+1}$. We thus obtain that:
$$\frac{\epsilon_t^{s+1}}{\delta_t^{s-1}(\epsilon_{t+1}^s + \rho_{t+1}^s)} \geq c_7 d^{t+c_0(N)} N^2 $$
where $c_7 = \frac{c_2 c_6^{s+1}}{c_3^{s-1}d^s(c_6^s + c_5^s)}$ is a positive constant independent of $N$ and $t$.
Since $\displaystyle\frac{\delta_t + \sigma_t}{d^{t+c_0(N)} N^2} = \frac{c_3}{d^{t+c_0(N)}N} + \frac{c_4}{N} \leq \frac{c_3 + c_4}{N}$, we have that $\delta_t + \sigma_t = \underset{N \to \infty}{o}(d^t N^2) $ uniformly in $t$. For a sufficiently large $N$, it follows that $c(\delta_t + \sigma_t) \leq c_7 d^t N^2$, which thus satisfies the last condition of \color{blue} theorem \ref{philippon} \color{black} for $P_t = R_{N,t}$.
\end{proof}

\subsection{Two preliminary lemmas}\text{ }

\begin{lm}{\label{lemme construction R_N}}
Let $f_1(z), . . . , f_s(z) \in \K[\![z]\!]$ be algebraicly independent function over $\K(z)$ and $N$ be an integer such that $N > s!$.
Then there exists a nonzero polynomial $R_N(z,X_1,...,X_{s}) \in \K[z,X_1,...,X_{s}]$ verifying the following conditions:\\
(1) $\deg_z(R_N) \leq N$;\\
(2) $\deg_X(R_N) \leq N$;\\
(3) $n_N := ord_{z=0}(R_N(z,f_1(z),...,f_{s}(z))) \geq \frac{N^{s+1}}{s!} := c_2 N^{s+1}$.
\end{lm}

\begin{proof}
Let us view $R_N$ as a polynomial in the $X_i$ with coefficients in $\K[z]$. By the degree bound (2), $R_N$ can be written as a sum of $\binom{s+N}{N}$ monomials with coefficients in $\K[z]$.    
Each of these coefficients is a polynomial in $z$ with coefficients in $\K$, of degree $\leq N$ by (1), and therefore has $(N+1)$ coefficients.

The functions $f_1(z),....,f_{s}(z)$ being algebraicly independent over $\K(z)$, the condition (3) amounts to consider a system with $(N+1)\binom{s+N}{N}$ unknowns and at most $\left\lceil \frac{N^{s+1}}{s!} \right\rceil $ equations.
Using the upper bound:
$$(N+1)\binom{s+N}{N} \geq \frac{(N+1)^{s+1}}{s!} > \frac{N^{s+1}}{s!} + 1 \geq \left\lceil \frac{N^{s+1}}{s!} \right\rceil $$

where the penultimate inequality follows from the assumption that $N > s!$. This system therefore has strictly more unknowns than equations, and thus has infinitely many solutions, in particular a non-trivial one.
\end{proof}

\begin{lm}{\label{lemme majoration hauteur denominateur alg}}
    Let $g \in \K(\!(z)\!)$ and $\alpha \in \overline{k}$ with $v_\infty(\alpha) > 0$ be such that $g(\alpha)$ converges. Moreover, assume that $g$ is algebraic over $k(z)$.
    Then there exists two positive real constants $\mathfrak{c}_1$ and $\mathfrak{c}_2$ such that for all $k \in \N$ :\\ 
• there exists a denominator $D_k$ of $g(\alpha^{d^k})$ such that $\deg_\theta(D_k) \leq \mathfrak{c}_1 d^k$;\\
• $\house{g(\alpha^{d^k})} \leq \mathfrak{c}_2d^k$.
\end{lm}
\begin{proof}
Since $g$ is algebraic, there exists $b_i(z) \in \F_q[\theta][z]$ such that
$$b_m(z)g^m(z) = \sum_{i=0}^{m-1} b_i(z) g^i(z).$$
First, note that this relation implies that for all $k \in \N$  that $g(\alpha^{d^k}) \in \overline{k}$, possibly after dividing if the polynomials $b_i(z)$ are all divisible by $z-\alpha^{d^k}$.

We may consider without loss of generalty that $k$ is sufficiently large so that $b_m(\alpha^{d^k}) \neq 0$, by increasing $\mathfrak{c}_1$ and $\mathfrak{c}_2$ if necessary.

We may assume that $\house{g(\alpha^{d^k})} \geq 0$. 
Set $\delta = \max \{ \deg_z(b_i(z)), i \in [\![0,m]\!] \}$, and note that the houses of each coefficients of the $b_i$ are bounded above by a constant $\mathfrak{c}_3$ defined as the maximum of their degrees in $\theta$. 
Thus, we have:
$$m \house{g(\alpha^{d^k})} = \house{\frac{1}{b_m(\alpha^{d^k})}\sum_{i=0}^{m-1} b_i(\alpha^{d^k}) g^i(\alpha^{d^k})}$$
$$\leq \house{\frac{1}{b_m(\alpha^{d^k})}}  + \max \left\{\house{b_i(\alpha^{d^k})} + i\house{g(\alpha^{d^k})}, i\in [\![0,m-1]\!]\right\}$$

Consequently,
$$\house{g(\alpha^{d^k})} \leq \house{\frac{1}{b_m(\alpha^{d^k})}} + \mathfrak{c}_3 + \max\{\house{\alpha},0\} \delta d^k.$$

It remains to bound $\house{\frac{1}{b_m(\alpha^{d^k})}}$.
Set $b_m(z) = \sum_{j=0}^\delta \beta_j z^j$.

 $$\house{\frac{1}{b_m(\alpha^{d^k})}} = \max\left\{v_\infty\Big( \sum_{j=0}^\delta \beta_j \sigma(\alpha)^{jd^k} \Big), \sigma \in \Hom_k(k(\alpha), \C_\infty) \right\}.$$

Let $\sigma$ be the automorphism which provides an upper bound for the previous expression. 
For $k$ sufficiently large (such that $\max\{|v_\infty(\beta_i) - v_\infty(\beta_j)|, (i,j) \in [\![0,\delta]\!]^2\} < |v_\infty(\sigma(\alpha))d^k|$), we have:
$$v_\infty\Big( \sum_{j=0}^\delta \beta_j \sigma(\alpha)^{jd^k} \Big) = v_\infty(\beta_0) \text{ if } v_\infty(\sigma(\alpha)) > 0;$$
$$v_\infty\Big( \sum_{j=0}^\delta \beta_j \sigma(\alpha)^{jd^k} \Big) = v_\infty(\beta_\delta) + v_\infty(\sigma(\alpha))\delta d^k < 0 \text{ if } v_\infty(\sigma(\alpha)) \leq 0.$$
In both cases, we obtain a positive constant $\mathfrak{c}_4$ such that for all $k \geq 0$,
$$\house{\frac{1}{b_m(\alpha^{d^k})}} \leq \mathfrak{c}_4.$$
Hence, for all $k \geq 0$, the following bound holds:
$$\house{g(\alpha^{d^k})} \leq \mathfrak{c}_2 d^k.$$

We now turn to the denominators. From the algebraic relation, we deduce that $b_m(z)g(z)$ is integral over $k(z)$:
$$(b_m(z)g(z))^m = \sum_{i=0}^{m-1} b_i(z)b_m(z)^{m-1-i} (b_m(z)g(z))^i.$$
Let $D_\alpha$ denote a denominator of $\alpha$. It folows that $D_\alpha^{\delta d^k} b_m(\alpha^{d^k})g(\alpha^{d^k})$ is integral over $k(\alpha)$:
$$(D_\alpha^{\delta d^k} b_m(\alpha^{d^k})g(\alpha^{d^k}))^m = \sum_{i=0}^{m-1} D_\alpha^{(m-i)\delta d^k} b_i(\alpha^{d^k}) b_m(\alpha^{d^k})^{m-1-i} (D_\alpha^{\delta d^k} b_m(\alpha^{d^k})g(\alpha^{d^k}))^i.$$
We deduce that $D_\alpha^{\delta d^k} D\left(\frac{1}{b_m(\alpha^{d^k})}\right)$ is a denominator of $g(\alpha^{d^k})$, where $D\left(\frac{1}{b_m(\alpha^{d^k})}\right)$ denotes a denominator of $\frac{1}{b_m(\alpha^{d^k})}$. The remaining task is to choose this last denominator appropriately.




The expression $$D_\alpha^{[k(\alpha):k]\delta d^k} \Norm_{k(\alpha)/k}(b_m(\alpha^{d^k}))$$
is a denominator of $\frac{1}{b_m(\alpha^{d^k})}$. It remains to compute its degree:
$$\deg_\theta\left(\Norm_{k(\alpha)/k}(b_m(\alpha^{d^k}))\right) = [k(\alpha):k]_{ins} \sum_{\sigma \in \Hom_k(k(\alpha), \C_\infty)} -v_\infty(b_m(\sigma(\alpha)^{d^k}) $$
$$\leq [k(\alpha):k] \max\left\{ -v_\infty(b_m(\sigma(\alpha)^{d^k}), \sigma \in \Hom_k(k(\alpha), \C_\infty)\right\}$$
$$= [k(\alpha):k] \house{b_m(\alpha^{d^k})}$$
$$\leq [k(\alpha):k] (\mathfrak{c}_3 + \max\{ \house{\alpha},0\} \delta d^k).$$
The denominator thus constructed therefore has degree at most
$$[k(\alpha):k]\mathfrak{c}_3 + (\deg_\theta(D_\alpha) + \max\{\house{\alpha},0\})[k(\alpha):k]\delta d^k := \mathfrak{c}_1d^k.$$
\end{proof}

\subsection{Proof of the theorem}\text{ }

Let us now adopt the notations from \color{blue}theorem \ref{Adam_Denis2}\color{black}. Possibly after extending $\K$, we may assume that $\alpha$ lies in $\K$. We denote:\\
$l = degtr_{k}\{f_1(\alpha),...,f_n(\alpha)\}$;\\
$s = degtr_{k(z)}\{f_1(z),...,f_n(z)\}$.\\
The implication $l \leq s$ of the equality follows immediately: a relation among the functions induces, via evaluation, a relation among their values (possibly after dividing the relation by $z-\alpha$). It therefore remains to prove that $l \geq s$. Hence, we may assume that $s$ is nonzero.

Reusing the notations from \color{blue}corollary \ref{cor philippon}\color{black}, we set $$\omega = (\omega_0, \omega_1,...,\omega_n) = (1, f_1(\alpha),...,f_n(\alpha)).$$ It therefore suffices to construct a sequence $(n_N)_{N \in \N}$ and polynomials $P_{N,t}$ satisfying our assumptions. Up to reordering, we may assume that $f_1(z),....,f_{s}(z)$ are algebraicly independent over $\K(z)$. Let $\deg_X$ denote the total degree in the $X_1,...,X_{s}$, and consider $N> s!$ an integer.

Let $R_N$ be a polynomial given by \color{blue} lemma  \ref{lemme construction R_N}\color{black}. The sequence $(n_N)_{N \in \N}$ defined by this lemma satisfies the assumption of \color{blue} corollary \ref{cor philippon}\color{black}. Thus, it remains to construct the polynomials. We set:
$$E_N(z) = R_N(z,f_1(z),...,f_{s}(z)) = \sum_{j=n_N}^\infty e_{j,N}z^j.$$
The function thus defined is nonzero, since $f_1(z),....,f_{s}(z)$ are algebraicly independent over $\K(z)$. 
Up to multiply $R_N$ by a constant, we may assume $e_{n_N,N} = 1$.
We denote $R_{N,0}(z,X_0,...,X_n) \in \K[z,X_0,...X_n]$ the homogeneous polynomial of total degree $N$ in the $X_0,...X_n$ satisfying:
$$R_{N,0}(z,1,X_1,...,X_n) = R_N(z,X_1,...,X_{s}).$$

We will construct all the $R_{N,t}$ by induction on $t \in \N$. We denote by $A_i(z)$ the $i$-th row of the matric $A(z)$, $\overline{f}(z) = (f_1(z),...f_n(z))^t$, $\overline{X} = (X_1,...,X_n)^t$ and $<\cdot,\cdot>$ the usual inner product.
We can express:
$$R_{N,0}(z^d,1,f_1(z^d),...,f_n(z^d)) =  R_{N,0}(z^d,1,<A_1(z), \overline{f}(z)>,...,<A_n(z), \overline{f}(z)>).$$

For all $i \in [\![1,n]\!]$, we set $U_{0}(\overline{X}) = \overline{X}$, and we define by induction for all $k \in \N^*$, 
$$U_{k}(\overline{X}) := (U_{k,1}(\overline{X}),...,U_{k,n}(\overline{X}))^t = A(\alpha^{d^{k-1}}) U_{k-1}(\overline{X}).$$

We then set for all $t \in \N^*$
\begin{equation}
  R_{N,t}(z,X_0,...,X_n) = R_{N,0}(z^{d^t},X_0,U_{t,1}(\overline{X}),...,U_{t,n}(\overline{X})).
  \label{Fernandes recurrence}
\end{equation}

The polynomials $R_{N,t}$ thus constructed all are homogeneous and of total degree $N$ in the $X_0,X_1,...,X_n$. We may now set $c_3 = 1$.\\

By evaluating $R_{N,t}$, we obtain:
$$R_{N,t}(\alpha,1,f_1(\alpha),...,f_n(\alpha)) = E_N(\alpha^{d^t}).$$

Moreover, for all $N$, there exists an element $c_0(N)$ such that for all $t \geq c_0(N)$, $|E_N(\alpha^{d^t})|_\infty = |\alpha|_\infty^{n_Nd^t}$, therefore
$E_N(\alpha^{d^t}) \neq 0$, or, setting $c_5 = c_6 =\log|\alpha|_\infty$:
$$-c_5 d^tn_N \leq \log|E_N(\alpha^{d^t})| \leq -c_6 d^t n_N.$$
We finally obtain:
$$-c_5d^tn_N \leq \log|R_{N,t}(\alpha,\omega_0,...,\omega_n)| \leq -c_6 d^tn_N.$$

All that remains is to control the height of the polynomials.\\

\begin{remarque}
By the construction, we may write:
$$R_{N,t}(\alpha,X_0,...,X_n) = R_{N,0}(\alpha^{d^t},X_0,U_{t,1},...,U_{t,n}),$$
where the $U_{t,i}$, $i\in [\![1,n]\!]$ are linear forms in the $X_1$,...,$X_n$, whose coefficients are all product of elements of the form $\prod_{k=0}^{t-1} \beta_k$,
where $\beta_k \in S_k$ for all $k \in [\![0,t-1]\!]$.  
\end{remarque}

It follows from this remark and our assumption that all the $P_{N,t}$ lies in $\L[X_0,...,X_n]$, where $\L$ is a finite extension of $\K$.

For each $t \in \N$, let us set
$$P_{N,t}(X_0,...,X_n) = R_{N,t+c_0(N)}(\alpha,X_0,...,X_n) \in \K[X_0,...,X_n].$$
These are indeed homogeneous polynomials in $\L[X_0,...,X_n]$ of degree $N$ satisfying $$-c_5d^{t+c_0(N)} n_N \leq \log|P_{N,t}(\omega)| \leq -c_6 d^{t+c_0(N)} n_N.$$
It therefore only remain to prove the following lemma in order to apply \color{blue}corollary \ref{cor philippon} \color{black} and conclude:

\begin{lm}
By possibly increasing $c_0(N)$, we may assume that for all $t \geq 0$:
$$h(P_{N,t}) \leq c_4 d^{t+c_0(N)} N.$$    
\end{lm}
\begin{proof}

Let $S_k$ denote the set of all the coefficients of the matrix $A(\alpha^{d^k}).$
From \color{blue} lemma \ref{lemme majoration hauteur denominateur alg}\color{black}, there exists two real positive constants $\mathfrak{c}_1$ and $\mathfrak{c}_2$ such that for all $k \in \N$:\\ 
• there exists $D_k$ a denominator of the family $S_k$ such that $\deg_\theta(D_k) \leq \mathfrak{c}_1 d^k$;\\
• for all $\gamma\in S_k$, $\house{\gamma} \leq \mathfrak{c}_2d^k$.

From the previous remark, $\Tilde{D}_t := \prod_{k=0}^{t+c_0(N)-1} D_k$ is a denominator of the coefficients of the $U_{t+c_0(N),i}$, whose degree satisfy:
$$\deg_\theta(\Tilde{D}_t) \leq \sum_{k=0}^{t+c_0(N)-1} \mathfrak{c}_1d^k = \mathfrak{c}_1 \frac{d^{t+c_0(N)}-1}{d-1} \leq \mathfrak{c}_1 d^{t+c_0(N)}.$$
We also have for all $i\in[\![1,n]\!]$ and all coefficients $S$ in $U_{t+c_0(N),i}$ the following bound:
$$\house{S} \leq \sum_{k=0}^{t+c_0(N)-1} \mathfrak{c}_2 d^k \leq \mathfrak{c}_2 d^{t+c_0(N)}.$$

Let us now focus on $R_{N,0}$. Recall that it is a polynomial in $z$ of degree at most $N$, and also a homogeneous polynomial in $X_0,...,X_n$ of degree $N$. Being a polynomial, it has only finitely many coefficients, all belonging to $\K$. We denote $M_R$ the maximum of the houses of its coefficients, and $D_R$ a denominator of all its coefficients. Observe that these two quantitites depend on the parameter $N$, but are independent of $t$.
We denote $D_\alpha$ a denominator of $\alpha$.

The coefficients of $P_{N,t}(X_0,...,X_n)$ thus admit the polynomial 
$D_RD_\alpha^{Nd^{t+c_0(N)}}\Tilde{D}_t^N$ as a denominator, whose degree is
$$\deg_\theta\left(D_RD_\alpha^{Nd^{t+c_0(N)}}\Tilde{D}_t^N\right) \leq \deg_\theta(D_R) + (\deg_\theta(D_\alpha) + \mathfrak{c}_1)Nd^{t+c_0(N)},$$
and their house are bounded by
$$M_R + (\max\{\house{\alpha},0\} + \mathfrak{c}_2)Nd^{t+c_0(N)}.$$

From \color{blue}proposition \ref{inegalite hauteur maison denominateur}\color{black}, we thus have 
$$h(P_{N,t}) \leq \deg_\theta(D_R) + \max\{0,M_R\} + (\deg_\theta(D_\alpha) + \max\{\house{\alpha},0\} + \mathfrak{c}_1 + \mathfrak{c}_2)Nd^{t+c_0(N)}$$
After increasing $c_0(N)$ if necessary, we have $\deg_\theta(D_R) + \max\{0,M_R\} \leq c_0(N)$. We finally deduce:
$$h(P_{N,t})\leq c_4d^{t+c_0(N)}N.$$
\end{proof}

We proceed to verify all the conditions of \color{blue} corollary $\ref{cor philippon}$ \color{black}, which lets us conclude the proof.

\newpage
\bibliographystyle{plain}
\bibliography{bibliography}
\end{document}